\documentclass[10pt]{article}
\usepackage{amsmath}
\numberwithin{equation}{section}
\usepackage{amsfonts}
\usepackage{amssymb}
\usepackage{graphicx}
\usepackage{mathrsfs}
\usepackage{xcolor}
\usepackage{verbatim}
\usepackage{mathrsfs}
\usepackage[body={15.5cm,21cm}, top=3cm]{geometry}
\usepackage{paralist}
\usepackage{ntheorem}
\usepackage{appendix}

\DeclareMathOperator*{\esssup}{ess\,sup}
\allowdisplaybreaks[4]
\usepackage{hyperref}
\hypersetup{colorlinks=true,
linkcolor=blue,
anchorcolor=blue,
citecolor=blue}
\providecommand{\U}[1]{\protect \rule{.1in}{.1in}}
\newtheorem{theorem}{Theorem}[section]

\newtheorem{corollary}[theorem]{Corollary}

\newtheorem{definition}[theorem]{Definition}

\newtheorem{lemma}[theorem]{Lemma}

\newtheorem{remark}[theorem]{Remark}

\theoremstyle{empty}

\newenvironment{proof}[1][Proof]{\noindent \textbf{#1.} }{\  \rule{0.5em}{0.5em}}
\begin{document}
\title{Non-homogeneous stochastic linear-quadratic optimal control problems with
multi-dimensional state and regime switching}
\author{Yuyang Chen\thanks{School of Mathematical Sciences, Shanghai Jiao Tong University, China (cyy0032@sjtu.edu.cn)}
\and
Peng Luo \thanks{School of Mathematical Sciences, Shanghai Jiao Tong University, China (peng.luo@sjtu.edu.cn). Financial support
from the National Natural Science Foundation of China (Grant No. 12101400) is gratefully acknowledged.}}

\maketitle
\begin{abstract}
In this paper, we study non-homogeneous stochastic linear-quadratic (LQ) optimal control problems with multi-dimensional state and regime switching. We focus on the corresponding stochastic Riccati equation, which is the same as that one in homogeneous stochastic LQ optimal control problem, and the adjoint backward stochastic differential equation (BSDE), which arises from the non-homogeneous terms in the state equation and cost functional. Both stochastic Riccati equation and adjoint BSDE are solved by the contraction mapping method, and are used to represent the closed-loop optimal control and the optimal value of our problems. 
\end{abstract}


\textbf{MSC-classification}: 93H20, 60H30.
\section{Introduction}
LQ optimal control problem is a widely researched topic and regarded as one of the most classical control problems. Wonham studied the stochastic LQ optimal control problem with deterministic coefficients in \cite{Wonham 1967} and discussed the corresponding Riccati equation in detail in \cite{Wonham 1968}. Bismut \cite{Bismut 1976} first focused on the stochastic Riccati equation but was only able to solve some special cases. Peng \cite{Peng 1992} used Bellman's principle and Wonham's method to obtain the existence and uniqueness results for a more general stochastic Riccati equation. After extensive research (see \cite{Kohlmann and Tang 2001, Kohlmann and Tang 2002, Kohlmann and Tang 2003a, Kohlmann and Tang 2003b}), Tang \cite{Tang 2003} effectively resolved the most comprehensive stochastic Riccati equation. We refer to \cite{Sun et al. 2016, Lü 2019, Sun et al. 2021} for further advancements in stochastic LQ optimal control problems and stochastic Riccati equations.

In engineering, management and economics problems, different states often require the establishment of different equations, leading to the emergence of the LQ problem with Poisson jumps. Wu and Wang \cite{Wu and Wang 2003} proposed the stochastic LQ problem whose state equation is driven by both Brownian motion and Poisson process, and solved the related Riccati equation whose coefficients partially degenerates to zero. Hu and {\O}ksendal \cite{Hu and Oksendal 2008} first studied the non-homogeneous stochastic LQ problem with Poisson jumps although they didn't solve the Riccati equations. Yu \cite{Yu 2017} obtained the existence and uniqueness result for forward-backward stochastic equations (FBSDEs) with Poisson jumps under monotone condition, and applied it to backward stochastic LQ problems. Li et al. \cite{Li et al. 2018} discussed the stochastic LQ problem with Poisson processes under the indefinite case.

Another form of state switching under varying equations is referred to as regime switching. Li and Zhou \cite{Li and Zhou 2002} studied a stochastic LQ problem in the finite time horizon with Markovian jumps in parameter values while Li et al. \cite{Li et al. 2003} discussed the case in the infinite time horizon. Liu et al. \cite{Liu et al. 2005} examined the near-optimal controls of regime switching LQ problems that come with indefinite control weight costs. Recently, there have been further breakthroughs concerning this issue. For example, Zhang et al. \cite{Zhang et al. 2021} established both open-loop and closed-loop solvability for stochastic LQ optimal control problem of Markovian regime switching system with deterministic coefficients. Hu et al. \cite{Hu et al. 2022a} studied the stochastic LQ control problem with regime switching, random coefficients and cone control constraint in which the state equation is one-dimensional. Wen et al. \cite{Wen et al. 2023}, Chen and Luo \cite{Chen and Luo 2023} used different methods to solve the multi-dimensional stochastic LQ optimal control problems with regime switching and the corresponding matrix-valued Riccati equations. The LQ problem with regime switching has numerous applications in the field of mathematical finance, especially mean-variance portfolio selection problems and asset-liability management, see \cite{Zhou and Yin 2003, Yin and Zhou 2004, Xie 2009, Shen et al. 2020, Hu et al. 2022b}.

In this paper, we are interested in solving non-homogeneous stochastic LQ optimal control problems with multi-dimensional state and regime switching. Our first contribution is to solve the matrix-valued stochastic Riccati equation, which is also discussed in the homogeneous problem \cite{Chen and Luo 2023}. In our previous work \cite{Chen and Luo 2023}, we constructed a Piccard iteration scheme and was only able to obtain the existence of the solution for the Riccati equation, while through this article we use the contraction mapping method to solve the Riccati equation. Our second contribution is to solve the adjoint BSDE, which comes from the non-homogeneous terms in the state equation and cost functional. The coefficients of this BSDE are all unbounded as well as Y is multi-dimensional, which makes measure transformation infeasible. Our last main contribution is to show the closed-loop solution and the optimal value for the non-homogeneous stochastic LQ problem, which are represented by the solution of the stochastic Riccati equation and the adjoint BSDE.

The rest of this paper is organized as follows. In Section 2, we formulate the multi-dimensional stochastic LQ optimal control problems with regime switching and give our three main results. The proofs of these theorems are presented sequentially in Section 3.

\section{Formulation of the problem and main results}
Let $(\Omega, \mathcal{F}, \mathbb{P})$ be a fixed complete probability space on which are defined a standard one-dimensional Brownian motion $W=\{W(t);0\leq t<\infty\}$ and a continuous-time stationary Markov chain $\alpha_{t}$ valued in a finite state space $\mathcal{M}=\{1,2, \ldots, \ell\}$ with $\ell>1$. We assume $W(t)$ and $\alpha_{t}$ are independent processes. The Markov chain has a generator $Q=\left(q_{i j}\right)_{\ell \times \ell}$ with $q_{i j} \geq 0$ for $i \neq j$ and $\sum_{j=1}^{\ell} q_{i j}=0$ for every $i \in \mathcal{M}$. Define the filtrations $\mathcal{F}_{t}=\sigma\left\{W(s), \alpha_{s}: 0 \leq s \leq t\right\} \vee \mathcal{N}$ and $\mathcal{F}_{t}^{W}=\sigma\{W(s): 0 \leq s \leq t\} \vee \mathcal{N}$, where $\mathcal{N}$ is the totality of all the $\mathbb{P}$-null sets of $\mathcal{F}$. For a random variable $\eta$, $\|\eta\|_{\infty}$ denotes the $L^{\infty}$-norm of $\eta$, i.e., $\|\eta\|_{\infty}:=\mathop{\esssup}\limits_{\omega}|\eta(\omega)|$. Equalities and inequalities between random variables and processes are understood in the $P$-a.s. and $P\otimes dt$-a.e. sense, respectively.

We use the following notation throughout the paper:
\begin{align*}
&\mathbb{R}^{n}:\text{ the }n\text{-dimensional Euclidean space with the Euclidean norm }|\cdot|;\\
&\mathbb{R}^{m\times n}:\text{ the Euclidean space of all }(m\times n)\text{ real matrices};\\
&\mathbb{S}^{n}:\text{ the space of all symmetric }(n\times n)\text{ real matrices};\\
&I_{n}:\text{ the identity matrix of size }n;\\
&M^{\top}:\text{ the transpose of a matrix }M;\\
&tr(M):\text{ the trace of a matrix }M;\\
&\langle\cdot,\cdot\rangle:\text{ the Frobenius inner product on }\mathbb{R}^{n\times m},\text{ which is defined by }\langle A,B\rangle=tr(A^{\top}B);\\
&|M|:\text{ the Frobenius norm of a matrix }M,\text{ defined by }\left(tr(MM^{\top})\right)^{\frac{1}{2}};
\end{align*}
Furthermore, we introduce the following spaces of random processes: for Euclidean space $\mathbb{H}=\mathbb{R}^{n},\mathbb{R}^{n\times m},\mathbb{S}^{n}$, $p\geq2$ and all $\{\mathcal{F}_{t}\}_{t\geq0}$-stopping time $\tau\leq T$,
\begin{align*}
&L_{\mathcal{F}}^{\infty}(\Omega;\mathbb{H})=\left\{\xi:\Omega\rightarrow\mathbb{H}\mid\xi\text{ is }\mathcal{F}_{T}\text{-measurable, and essentially bounded }\right\};\\
&\begin{aligned}\mathcal{H}_{\mathcal{F}}^{p}(0,T;\mathbb{H})=
  &\Bigg\{\phi:[0,T]\times \Omega\rightarrow\mathbb{H}\mid\phi(\cdot)\text{ is an }\left\{\mathcal{F}_{t}\right\}_{t\geq0}\text{-adapted process}\\
  &\text{ with }\mathbb{E}\left(\int_{0}^{T}|\phi(t)|dt\right)^{p}<\infty\Bigg\};
  \end{aligned}\\
&\begin{aligned}L_{\mathcal{F}}^{p}(0,T;\mathbb{H})=
&\Bigg\{\phi:[0,T]\times \Omega\rightarrow\mathbb{H}\mid\phi(\cdot)\text{ is an }\left\{\mathcal{F}_{t}\right\}_{t\geq0}\text{-adapted process}\\
&\text{ with }\mathbb{E}\left(\int_{0}^{T}|\phi(t)|^{2}dt\right)^{\frac{p}{2}}<\infty\Bigg\};
\end{aligned}\\
&\begin{aligned}L_{\mathcal{F}}^{p}(\Omega;C(0,T;\mathbb{H}))=&\Bigg\{\phi:[0,T]\times\Omega\rightarrow\mathbb{H}\mid\phi(\cdot)\text{ is an }\left\{\mathcal{F}_{t}\right\}_{t \geq 0}\text{-adapted process}\\
&\text{ and continuous with }\mathbb{E}\left(\sup_{t\in[0,T]}|\phi(t)|^{p}\right)<\infty\Bigg\};\end{aligned}\\
&\begin{aligned}L_{\mathcal{F}}^{\infty}(0,T;\mathbb{H})=&\Bigg\{\phi:[0,T]\times\Omega\rightarrow\mathbb{H}\mid\phi(\cdot)\text{ is an }\left\{\mathcal{F}_{t}\right\}_{t\geq 0}\text{-adapted essentially bounded process}\Bigg\}.\end{aligned}\\
&\begin{aligned}L_{\mathcal{F}}^{2,\mathrm{bmo}}(0,T;\mathbb{H})=&\Bigg\{\phi:[0,T]\times\Omega\rightarrow\mathbb{H}\mid\phi(\cdot)\text{ is an }\mathcal{F}\text{-progressively measurable process}\\
&\text{ with }\sup_{0\leq\tau\leq T}\left\|\mathbb{E}\left[\int_{\tau}^{T}|\phi(s)|^{2}ds\mid\mathcal{F}_{\tau}\right]\right\|^{\frac{1}{2}}_{\infty}<\infty\Bigg\};\end{aligned}
\end{align*}
$L_{\mathcal{F}^W}^{\infty}(\Omega;\mathbb{H})$, $\mathcal{H}_{\mathcal{F}^W}^{p}(0,T;\mathbb{H})$, $L_{\mathcal{F}^W}^{p}(0,T;\mathbb{H})$, $L_{\mathcal{F}^W}^{p}(\Omega;C(0,T;\mathbb{H}))$, $L_{\mathcal{F}^W}^{\infty}(0,T;\mathbb{H})$ and $L_{\mathcal{F}^{W}}^{2,\mathrm{bmo}}(0,T;\mathbb{H})$ are defined in a same manner by replacing $\mathcal{F}$ by $\mathcal{F}^{W}$.

We now introduce the non-homogeneous multi-dimensional stochastic LQ optimal problem with regime switching. Consider the following $n$-dimensional controlled linear stochastic differential equation on the finite time interval $[0,T]$:
\begin{equation}\label{state}
\left\{\begin{array}{l}
\begin{aligned}
dX(t)=&\left[A\left(t,\alpha_{t}\right)X(t)+B\left(t,\alpha_{t}\right)u(t)+b\left(t,\alpha_{t}\right)\right]dt\\
&+\left[C\left(t,\alpha_{t}\right)X(t)+D\left(t,\alpha_{t}\right)u(t)+\sigma\left(t,\alpha_{t}\right)\right]dW(t),\quad t\in[0,T],\end{aligned}\\
X(0)=x, \alpha_{0}=i_{0}
\end{array}\right.
\end{equation}
where $A(t,\omega,i),B(t,\omega,i),C(t,\omega,i),D(t,\omega,i)$ are all $\left\{\mathcal{F}^{W}_{t}\right\}_{t\geq 0}$-adapted processes of suitable sizes for $i \in \mathcal{M}$ and $x\in\mathbb{R}^{n}$ is an initial state, $i_{0}\in\mathcal{M}$ is an initial regime. The solution $X=\{X(t);0\leq t\leq T\}$ of \eqref{state}, valued in $\mathbb{R}^{n}$, is called a state process; the process $u=\{u(t);0\leq t\leq T\}$ of \eqref{state}, valued in $\mathbb{R}^{m}$, is called a control which influences the state $X$, and is taken from the space $\mathcal{U}[0,T]:=L_{\mathcal{F}}^{2}(0,T;\mathbb{R}^{m})$.

In order to measure the performance of control $u(\cdot)$, we introduce the following quadratic cost functional:
\begin{equation}\label{cost}
\begin{aligned}
J\left(x, i_{0}, u(\cdot)\right):=\mathbb{E}&\left[\int_{0}^{T}\big(\langle Q(t,\alpha_{t})\left(X(t)-q(t,\alpha_{t})\right), X(t)-q(t,\alpha_{t})\rangle+\langle R(t,\alpha_{t})\left(u(t)-r(t,\alpha_{t})\right),\right.\\
&\qquad\qquad u(t)-r(t,\alpha_{t})\rangle\big)dt\left.+\left\langle G(\alpha_{T})\left(X(T)-g(\alpha_{T})\right), X(T)-g(\alpha_{T})\right\rangle\right].
\end{aligned}
\end{equation}

For state equation \eqref{state} and cost functional \eqref{cost}, we introduce the following assumption:
\begin{flushleft}
$(\mathscr{A}1)$ For all $i \in \mathcal{M}$,
\begin{align*}
\left\{\begin{array}{ll}
A(t,\omega,i),~C(t,\omega,i)\in L_{\mathcal{F}^{W}}^{\infty}(0,T;\mathbb{R}^{n\times n
}),\\
B(t,\omega,i),~D(t,\omega,i)\in L_{\mathcal{F}^{W}}^{\infty}(0,T;\mathbb{R}^{n\times m}),\\
Q(t,\omega,i)\in L_{\mathcal{F}^{W}}^{\infty}(0,T;\mathbb{S}^{n}),\\
R(t,\omega,i)\in L_{\mathcal{F}^{W}}^{\infty}(0,T;\mathbb{S}^{m}),\\
G(\omega,i)\in L_{\mathcal{F}^{W}}^{\infty}(\Omega;\mathbb{S}^{n}),\\
b(t,\omega,i),~\sigma(t,\omega,i),~q(t,\omega,i)\in L_{\mathcal{F}^{W}}^{\infty}(0,T;\mathbb{R}^{n
}),\\
r(t,\omega,i)\in L_{\mathcal{F}^{W}}^{\infty}(0,T;\mathbb{R}^{m
}),\\
g(\omega,i)\in L_{\mathcal{F}^{W}}^{\infty}(\Omega;\mathbb{R}^{n})
\end{array}\right.
\end{align*}
\end{flushleft}
Under condition $(\mathscr{A}1)$, for any initial state $x$ and any control $u(\cdot)\in\mathcal{U}[0,T]$, standard SDE theory shows that equation \eqref{state} has a unique solution $X(\cdot)\in L_{\mathcal{F}}^{2}(\Omega;C(0,T;\mathbb{R}^{n}))$. We call such $(X(\cdot), u(\cdot))$ an admissible pair.

Then the following problem, called stochastic linear-quadratic optimal control problem with regime switching, can be formulated.\\
\\
$\mathbf{Problem~(SLQ)}$. For any initial pair $\left(x, i_{0}\right)\in\mathbb{R}^{n}\times\mathcal{M}$, find a control $u^{*}\in\mathcal{U}[0,T]$ such that
\begin{equation}\label{SLQ}
J\left(x,i_{0},u^{*}\right)=\inf_{u\in\mathcal{U}[0,T]}J\left(x,i_{0},u\right)\equiv V\left(x,i_{0}\right).
\end{equation}
Any element $u^{*}\in\mathcal{U}[0,T]$ satisfying \eqref{SLQ} is called an optimal control of Problem (SLQ) corresponding to the initial pair $\left(x, i_{0}\right)\in\mathbb{R}^{n}\times\mathcal{M}$, the corresponding state process $X^{*}(\cdot)\equiv X(\cdot;u^{*})$ is called an optimal state process. We also call $V\left(x,i_{0}\right)$ the value function of Problem (SLQ). Our objective is to solve Problem (SLQ). In the classical LQ problem, it is commonly observed that the positive definite coefficients in \eqref{cost} serve as a sufficient condition for the solvability of the Riccati equation and subsequently the LQ problem, as illustrated in \cite{Tang 2003}. Therefore, we also introduce the following assumption.
\begin{flushleft}
$(\mathscr{A}2)$ For all $i \in \mathcal{M},t\in[0,T]$ and some $\lambda>0$,
\begin{align*}
R(t,i)\geq\lambda I_{m},\quad Q(t,i)\geq0,\quad G(i)\geq0.
\end{align*}
\end{flushleft}

\begin{remark}
When $R(\cdot,i)>0, Q(\cdot,i)-S(\cdot,i)^{\top}R(\cdot,i)^{-1}S(\cdot,i)\geq0,i\in\mathcal{M}$, the SLQ problem with the state equation
\begin{equation*}
\left\{\begin{array}{l}
\begin{aligned}
dX(t)=&\left[A\left(t,\alpha_{t}\right)X(t)+B\left(t,\alpha_{t}\right)u(t)+b\left(t,\alpha_{t}\right)\right]dt\\
&+\left[C\left(t,\alpha_{t}\right)X(t)+D\left(t,\alpha_{t}\right)u(t)+\sigma\left(t,\alpha_{t}\right)\right]dW(t),\quad t\in[0,T],\end{aligned}\\
X(0)=x, \alpha_{0}=i_{0}
\end{array}\right.
\end{equation*}
and the cost functional
\begin{align*}
J_{1}\left(x,i_{0},u(\cdot)\right)=&E\bigg[\langle G(\alpha_{T})\left(X(T)-g(\alpha_{T})\right), X(T)-g(\alpha_{T})\rangle\\
&\quad+\int_{0}^{T}\left\langle\left(\begin{array}{cc}
Q(t,\alpha_{t}) & S(t,\alpha_{t})^{\top} \\
S(t,\alpha_{t}) & R(t,\alpha_{t})
\end{array}\right)\left(\begin{array}{c}
X(t)-q(t,\alpha_{t})\\
u(t)-r(t,\alpha_{t})
\end{array}\right),\left(\begin{array}{c}
X(t)-q(t,\alpha_{t})\\
u(t)-r(t,\alpha_{t})
\end{array}\right)\right\rangle dt\bigg]
\end{align*}
is equivalent to another one with
\begin{equation*}
\left\{\begin{array}{l}
\begin{aligned}
dX(t)=&\left[\widetilde{A}\left(t,\alpha_{t}\right)X(t)+B\left(t,\alpha_{t}\right)\widetilde{u}(t)+b\left(t,\alpha_{t}\right)\right]dt\\
&+\left[\widetilde{C}\left(t,\alpha_{t}\right)X(t)+D\left(t,\alpha_{t}\right)\widetilde{u}(t)+\sigma\left(t,\alpha_{t}\right)\right]dW(t),\quad t\in[0,T],\end{aligned}\\
X(0)=x, \alpha_{0}=i_{0}
\end{array}\right.
\end{equation*}
and
\begin{align*}
J_{2}\left(x,i_{0},\widetilde{u}(\cdot)\right)=&E\bigg[\langle G(\alpha_{T})\left(X(T)-g(\alpha_{T})\right), X(T)-g(\alpha_{T})\rangle\\
&\quad+\int_{0}^{T}\left\langle\left(\begin{array}{cc}
\widetilde{Q}(t,\alpha_{t}) & 0 \\
0 & R(t,\alpha_{t})
\end{array}\right)\left(\begin{array}{c}
X(t)-q(t,\alpha_{t})\\
\widetilde{u}(t)-\widetilde{r}(t,\alpha_{t})
\end{array}\right),\left(\begin{array}{c}
X(t)-q(t,\alpha_{t})\\
\widetilde{u}(t)-\widetilde{r}(t,\alpha_{t})
\end{array}\right)\right\rangle dt\bigg]
\end{align*}
where
\begin{align*}
&\widetilde{A}(t,\alpha_{t})=A(t,\alpha_{t})-B(t,\alpha_{t})R(t,\alpha_{t})^{-1}S(t,\alpha_{t}),\quad\widetilde{C}(t,\alpha_{t})=C(t,\alpha_{t})-D(t,\alpha_{t})R(t,\alpha_{t})^{-1}S(t,\alpha_{t}),\\
&\widetilde{r}(t,\alpha_{t})=r(t,\alpha_{t})+R(t,\alpha_{t})^{-1}S(t,\alpha_{t})q(t,\alpha_{t}),\\
&\widetilde{Q}(t,\alpha_{t})=Q(t,\alpha_{t})-S(t,\alpha_{t})^{\top}R(t,\alpha_{t})^{-1}S(t,\alpha_{t}),\quad\widetilde{u}(t)=u(t)+R(t,\alpha_{t})^{-1}S(t,\alpha_{t})X(t),
\end{align*}
therefore we only need to consider $S(\cdot,i)=0,i\in\mathcal{M}$ in this paper.
\end{remark}

\subsection{Stochastic Riccati equation and adjoint BSDE}
Inspired by the classic methods on non-homogeneous LQ optimal control problems, we consider the following $n\times n$-dimensional stochastic Riccati equation (SRE for short):
\begin{equation}\label{Riccati}
\left\{
\begin{array}{l}
\begin{aligned}
dP(t,i)=&-\left[P(t,i)A(t,i)+A(t,i)^{\top}P(t,i)+C(t,i)^{\top}P(t,i)C(t,i)+\Lambda(t,i) C(t,i)\right.\\
&\quad+C(t,i)^{\top}\Lambda(t,i)+Q(t,i)+\textstyle\sum_{j=1}^{l}q_{ij}P(t,j)-\left(P(t,i)B(t,i)+C(t,i)^{\top}P(t,i)D(t,i)\right.\\
&\quad\left.+\Lambda(t,i)D(t,i)\right)\left(R(t,i)+D(t,i)^{\top}P(t,i)D(t,i)\right)^{-1}\left(B(t,i)^{\top}P(t,i)\right.\\
&\quad\left.\left.+D(t,i)^{\top}P(t,i)C(t,i)+D(t,i)^{\top}\Lambda(t,i)\right)\right]dt+\Lambda(t,i) dW(t),
\end{aligned}\\
R(t,i)+D(t,i)^{\top}P(t,i)D(t,i)>0,\quad t\in[0,T],\\
P(T,i)=G(i),\quad i\in\mathcal{M},
\end{array}
\right.
\end{equation}
and $n$-dimensional adjoint BSDE:
\begin{equation}\label{K}
\left\{
\begin{array}{l}
\begin{aligned}
dK(t,i)=-&\left[\big(A(t,i)-B(t,i)\Gamma(t,i)\big)^{\top}K(t,i)+\big(C(t,i)-D(t,i)\Gamma(t,i)\big)^{\top}L(t,i)\right.\\
&+\Gamma(t,i)^{\top}\big(D(t,i)^{\top}P(t,i)\sigma(t,i)-R(t,i)r(t,i)\big)+Q(t,i)q(t,i)-P(t,i)b(t,i)\\
&\left.-\big(C(t,i)^{\top}P(t,i)+\Lambda(t,i)\big)\sigma(t,i)+\textstyle\sum_{j=1}^{l}q_{ij}K(t,j)\right]dt+L(t,i)dW(t),\quad t\in[0,T],
\end{aligned}\\
K(T,i)=G(i)g(i),\quad i\in\mathcal{M},
\end{array}
\right.
\end{equation}
where
\begin{align*}
\Gamma(t,i)=\big(R(t,i)+D(t,i)^{\top}P(t,i)D(t,i)\big)^{-1}\big(B(t,i)^{\top}P(t,i)+D(t,i)^{\top}P(t,i)C(t,i)+D(t,i)^{\top}\Lambda(t,i)\big).
\end{align*}

\begin{definition}
A vector process $(P(\cdot,i),\Lambda(\cdot,i))_{i=1}^{\ell}$ is called a solution of SRE \eqref{Riccati}, if it satisfies \eqref{Riccati}, and $(P(\cdot,i),\Lambda(\cdot,i))\in L_{\mathcal{F}^{W}}^{\infty}(\Omega;C(0,T;\mathbb{S}^{n}))\times L_{\mathcal{F}^{W}}^{2,\mathrm{bmo}}(0,T;\mathbb{S}^{n})$ for all $i\in\mathcal{M}$.
\end{definition}

\begin{definition}
A vector process $(K(\cdot,i),L(\cdot,i))_{i=1}^{\ell}$ is called a solution of BSDE \eqref{K}, if it satisfies \eqref{Riccati}, and $(K(\cdot,i),L(\cdot,i))\in L_{\mathcal{F}^{W}}^{\infty}(\Omega;C(0,T;\mathbb{R}^{n}))\times L_{\mathcal{F}^{W}}^{2,\mathrm{bmo}}(0,T;\mathbb{R}^{n})$ for all $i\in\mathcal{M}$.
\end{definition}

Before introducing our main results, we recall the following lemma from standard matrix analysis, which is a direct consequence of \cite[Theorem 7.4.1.1]{Horn and Johnson 2012}. We will use this lemma throughout this paper.

\begin{lemma}\label{le-a}
Let $\mathbf{A},\mathbf{B}\in\mathbb{S}^{n}$ with $\mathbf{B}$ being positive semi-definite. Then with $\lambda_{max}(\mathbf{A})$ denoting the largest eigenvalue of $\mathbf{A}$, we have
\begin{align*}
tr(\mathbf{A}\mathbf{B})\leq\lambda_{max}(\mathbf{A})\cdot tr(\mathbf{B}).
\end{align*}
\end{lemma}
\begin{corollary}
Let $\mathbf{A}\in\mathbb{R}^{n\times m},\mathbf{B}\in\mathbb{R}^{m\times k}$. We have
\begin{align*}
|\mathbf{A}\mathbf{B}|\leq|\mathbf{A}|\cdot|\mathbf{B}|.
\end{align*}
\end{corollary}

\subsection{Main results}
In \cite{Chen and Luo 2023}, Chen and Luo showed the existence of the solution for SRE \eqref{Riccati}. Now we give the following theorem, which also gives the uniqueness of the solution for SRE \eqref{Riccati}.
\begin{theorem}\label{th-1}
Let $(\mathscr{A}1)$ and $(\mathscr{A}2)$ hold. There exists a sufficiently small constant $L_{\sigma}>0$, when
\begin{equation}\label{condition}
|D(t,i)R(t,i)^{-1}D(t,i)^{\top}|\leq L_{\sigma},\quad i\in\mathcal{M},~t\in[0,T],
\end{equation}
SRE \eqref{Riccati} has a unique solution $\left(P(\cdot,i),\Lambda(\cdot,i)\right)_{i=1}^{l}$ such that $\left(P(\cdot,i),\Lambda(\cdot,i)\right)\in L_{\mathcal{F}^{W}}^{\infty}(0,T;\mathbb{S}^{n})$\\
$\times L_{\mathcal{F}^{W}}^{2,\mathrm{bmo}}(0,T;\mathbb{S}^{n})$ and $P(\cdot,i)\geq0$ for all $i\in\mathcal{M}$.
\end{theorem}

Based on the solvability of SRE \eqref{Riccati}, we can propose the following theorem, showing the existence and uniqueness results for the solution of BSDE \eqref{K}, whose coefficients are unbounded.

\begin{theorem}\label{th-2}
Under the conditions of Theorem \ref{th-1}, BSDE \eqref{K} has a unique solution $\left(K(\cdot,i),L(\cdot,i)\right)_{i=1}^{l}$ such that $\left(K(\cdot,i),L(\cdot,i)\right)\in L_{\mathcal{F}^{W}}^{\infty}(0,T;\mathbb{R}^{n})\times L_{\mathcal{F}^{W}}^{2,\mathrm{bmo}}(0,T;\mathbb{R}^{n})$ for all $i\in\mathcal{M}$.
\end{theorem}

By leveraging the solvability of SRE \eqref{Riccati} and BSDE \eqref{K}, we can effectively address Problem (SLQ). The following theorem provides an optimal feedback control for Problem (SLQ), as well as the optimal value.

\begin{theorem}\label{th-3}
Under the conditions of Theorem \ref{th-1}, Problem (SLQ) admits an optimal control, as a feedback function of the time $t$, the market regime $i$ and the state $X$,
\begin{equation}\label{optimal}
\begin{aligned}
u^{*}(t,X,i)=&-\left(R(t,i)+D(t,i)^{\top}P(t,i)D(t,i)\right)^{-1}\\
&\cdot\left[\left(B(t,i)^{\top}P(t,i)+D(t,i)^{\top}P(t,i)C(t,i)+D(t,i)^{\top}\Lambda(t,i)\right)X\right.\\
&\quad+\left.D(t,i)^{\top}P(t,i)\sigma(t,i)-R(t,i)r(t,i)-B(t,i)^{\top}K(t,i)-D(t,i)^{\top}L(t,i)\right]
\end{aligned}
\end{equation}
Moreover, the corresponding optimal value is
\begin{align*}
V(x, i_{0})=&\langle P(0,i_{0})x,x\rangle-2\langle K(0,i_{0}),x\rangle+\mathbb{E}\big[\langle G(\alpha_{T})g(\alpha_{T}),g(\alpha_{T})\rangle\big]\\
&+\mathbb{E}\bigg[\int_{0}^{T}\big[\langle Q(t,\alpha_{t})q(t,\alpha_{t}),q(t,\alpha_{t})\rangle+\langle R(t,\alpha_{t})r(t,\alpha_{t}),r(t,\alpha_{t})\rangle+\langle P(t,\alpha_{t})\sigma(t,\alpha_{t}),\sigma(t,\alpha_{t})\rangle\\
&\quad-2\langle K(t,\alpha_{t}),b(t,\alpha_{t})\rangle-2\langle L(t,\alpha_{t}),\sigma(t,\alpha_{t})\rangle-\langle\big(R(t,\alpha_{t})+D(t,\alpha_{t})^{\top}P(t,\alpha_{t})D(t,\alpha_{t})\big)^{-1}\\
&\quad\cdot \big(D(t,\alpha_{t})^{\top}P(t,\alpha_{t})\sigma(t,\alpha_{t})-R(t,\alpha_{t})r(t,\alpha_{t})-B(t,\alpha_{t})^{\top}K(t,\alpha_{t})-D(t,\alpha_{t})^{\top}L(t,\alpha_{t})\big),\\
&\qquad D(t,\alpha_{t})^{\top}P(t,\alpha_{t})\sigma(t,\alpha_{t})-R(t,\alpha_{t})r(t,\alpha_{t})-B(t,\alpha_{t})^{\top}K(t,\alpha_{t})-D(t,\alpha_{t})^{\top}L(t,\alpha_{t})\rangle\big)dt\bigg],
\end{align*}
where $\left(P(\cdot,i),\Lambda(\cdot,i)\right)_{i=1}^{l}$ is the unique solution of \eqref{Riccati} and $\left(K(\cdot,i),L(\cdot,i)\right)_{i=1}^{l}$ is the unique solution of \eqref{K}.
\end{theorem}

\section{Proofs of our main results}
\subsection{Solvability of SRE \eqref{Riccati}}
In this subsection, we focus on the solvability of SRE \eqref{Riccati}. The key ingredient of our approach is the construction of a contraction map, which is inspired by \cite{Chen and Luo 2023}. 

\begin{proof}[Proof\ of\ Theorem \ref{th-1}]
First of all, for $i\in\mathcal{M},~t\in[0,T]$, we denote
\begin{equation*}\label{notation}
\begin{aligned}
&\Pi(t,i,P,\Lambda):=P(t,i)A(t,i)+A(t,i)^{\top}P(t,i)+C(t,i)^{\top}P(t,i)C(t,i)+\Lambda(t,i)C(t,i)+C(t,i)^{\top}\Lambda(t,i),\\
&H(t,i,P,\Lambda):=-\left(P(t,i)B(t,i)+C(t,i)^{\top}P(t,i)D(t,i)+\Lambda(t,i)D(t,i)\right)\\
&\quad\cdot\left(R(t,i)+D(t,i)^{\top}P(t,i)D(t,i)\right)^{-1}\left(B(t,i)^{\top}P(t,i)+D(t,i)^{\top}P(t,i)C(t,i)+D(t,i)^{\top}\Lambda(t,i)\right).
\end{aligned}
\end{equation*}
Consider $\left(p(\cdot,i),\lambda(\cdot,i)\right)_{i=1}^{l}\in L_{\mathcal{F}^{W}}^{\infty}(0,T;\mathbb{S}^{n})\times L_{\mathcal{F}^{W}}^{2,\mathrm{bmo}}(0,T;\mathbb{S}^{n\times d})$, $p(\cdot,i)\geq0$ and the following equation:
\begin{equation}\label{BSDE3}
\left\{
\begin{array}{l}
dP(t,i)=-[\Pi(t,i,P,\Lambda)+Q(t,i)+H(t,i,P,\Lambda)+q_{ii}P(t,i)+\sum_{j\neq i}q_{ij}p(t,j)]dt+\Lambda(t,i)dW(t),\\
R(t,i)+D(t,i)^{\top}P(t,i)D(t,i)>0,~t\in[0,T],~i\in\mathcal{M},\\
P(T,i)=G(i).
\end{array}
\right.
\end{equation}
From \cite{Chen and Luo 2023}, BSDE \eqref{BSDE3} has a unique solution $\left(P(\cdot,i),\Lambda(\cdot,i)\right)_{i=1}^{l}\in L_{\mathcal{F}^{W}}^{\infty}(0,T;\mathbb{S}^{n})\times L_{\mathcal{F}^{W}}^{2,\mathrm{bmo}}(0,T;\mathbb{S}^{n\times d})$ and $P(\cdot,i)\geq0$. Now we are going to prove $\Theta\big((p(\cdot,i),\lambda(\cdot,i))_{i=1}^{l}\big):=(P(\cdot,i),\Lambda(\cdot,i))_{i=1}^{l}$ is a contraction map on the closed convex set $\mathscr{C}$ by
\begin{equation*}
\begin{aligned}
\mathscr{C}:=&\bigg\{(p(\cdot,i),\lambda(\cdot,i))_{i=1}^{l}\in L_{\mathcal{F}^{W}}^{\infty}(0,T;\mathbb{S}^{n})\times L_{\mathcal{F}^{W}}^{2,\mathrm{bmo}}(0,T;\mathbb{S}^{n}):p(\cdot,i)\geq0,~\|e^{\frac{1}{2}\rho\cdot}p(\cdot,i)\|^{2}_{L_{\mathcal{F}^{W}}^{\infty}(0,T;\mathbb{S}^{n})}\\
&\qquad+\frac{1}{2}\|e^{\frac{1}{2}\cdot}\lambda(\cdot,i)\|^{2}_{L_{\mathcal{F}^{W}}^{2,\mathrm{bmo}}(0,T;\mathbb{S}^{n})}\leq2\left[\|e^{\frac{1}{2}\rho T}G(i)\|^{2}_{L_{\mathcal{F}^{W}}^{\infty}(\Omega;\mathbb{S}^{n})}+\|e^{\frac{1}{2}\rho\cdot}Q(\cdot,i)\|^{2}_{L_{\mathcal{F}^{W}}^{\infty}(0,T;\mathbb{R}^{n})}\right]\bigg\}
\end{aligned}
\end{equation*}
where $\rho$ is a positive constant determined later. Applying It\^o's formula to $e^{\rho t}|P(t,i)|^{2}$ and taking conditional expectation, we have
\begin{align*}
&e^{\rho t}|P(t,i)|^{2}+\mathbb{E}_{t}\left[\int_{t}^{T}e^{\rho s}|\Lambda(s,i)|^{2}ds\right]\\
&\quad=\mathbb{E}_{t}\left[e^{\rho T}|P(T,i)|^{2}\right]+2\mathbb{E}_{t}\left[\int_{t}^{T}\langle\Pi(s,i,P,\Lambda)+Q(s,i)+H(s,i,P,\Lambda)+q_{ii}P(s,i)\right.\\
&\qquad\qquad\qquad\qquad\qquad\qquad\left.+\sum_{j\neq i}q_{ij}p(s,j),P(s,i)\rangle ds\right]-\rho\mathbb{E}_{t}\left[\int_{t}^{T}e^{\rho s}|P(s,i)|^{2}ds\right]\\
&\quad=\mathbb{E}_{t}\left[e^{\rho T}|P(T,i)|^{2}\right]+2L_{A}\mathbb{E}_{t}\left[\int_{t}^{T}e^{\rho s}|P(s,i)|^{2}ds\right]+2L_{C}\mathbb{E}_{t}\left[\int_{t}^{T}e^{\rho s}|P(s,i)|^{2}ds\right]\\
&\qquad+\frac{1}{2}\mathbb{E}_{t}\left[\int_{t}^{T}e^{\rho s}|\Lambda(s,i)|^{2}ds\right]+\mathbb{E}_{t}\left[\int_{t}^{T}e^{\rho s}|P(s,i)|^{2}ds\right]+\mathbb{E}_{t}\left[\int_{t}^{T}e^{\rho s}|Q(s,i)|^{2}ds\right]\\
&\qquad+2(l-1)^{2}q^{2}T\mathbb{E}_{t}\left[\int_{t}^{T}e^{\rho s}|P(s,i)|^{2}ds\right]+\frac{1}{2(l-1)T}\sum_{j\neq i}\mathbb{E}_{t}\left[\int_{t}^{T}e^{\rho s}|p(s,j)|^{2}ds\right]\\
&\qquad-\rho\mathbb{E}_{t}\left[\int_{t}^{T}e^{\rho s}|P(s,i)|^{2}ds\right]\\
&\quad=\|e^{\frac{1}{2}\rho T}G(i)\|^{2}_{L_{\mathcal{F}^{W}}^{\infty}(\Omega;\mathbb{S}^{n})}+\|e^{\frac{1}{2}\rho\cdot}Q(\cdot,i)\|^{2}_{L_{\mathcal{F}^{W}}^{\infty}(0,T;\mathbb{R}^{n})}+\frac{1}{2(l-1)}\sum_{j\neq i}\|e^{\frac{1}{2}\rho\cdot}p(\cdot,j)\|^{2}_{L_{\mathcal{F}^{W}}^{\infty}(0,T;\mathbb{S}^{n})}\\
&\qquad+\bigg(2L_{A}+2L_{C}+2(l-1)^{2}q^{2}T+1-\rho\bigg)\mathbb{E}_{t}\left[\int_{t}^{T}e^{\rho s}|P(s,i)|^{2}ds\right]+\frac{1}{2}\mathbb{E}_{t}\left[\int_{t}^{T}e^{\rho s}|\Lambda(s,i)|^{2}ds\right]
\end{align*}
where $q=\max\limits_{j\neq i}q_{ij}$ and $L_{A},L_{C}$ are positive constants depending on $A(\cdot,i),C(\cdot,i)$. Now we choose $\rho=2L_{A}+2L_{C}+2(l-1)^{2}q^{2}T+1$ and obtain $(P(\cdot,i),\Lambda(\cdot,i))_{i=1}^{l}\in\mathscr{C}$ from $(p(\cdot,i),\lambda(\cdot,i))_{i=1}^{l}\in\mathscr{C}$. For $(p_{1}(\cdot,i),\lambda_{1}(\cdot,i))_{i=1}^{l},(p_{2}(\cdot,i),\lambda_{2}(\cdot,i))_{i=1}^{l}\in\mathscr{C}$, we denote
\begin{align*}
(P_{1}(\cdot,i),\Lambda_{1}(\cdot,i))_{i=1}^{l}=\Theta\big((p_{1}(\cdot,i),\lambda_{1}(\cdot,i))_{i=1}^{l}\big),~(P_{2}(\cdot,i),\Lambda_{2}(\cdot,i))_{i=1}^{l}=\Theta\big((p_{2}(\cdot,i),\lambda_{2}(\cdot,i))_{i=1}^{l}\big)
\end{align*}
and
\begin{align*}
&\Delta P(\cdot,i)=P_{1}(\cdot,i)-P_{2}(\cdot,i),~\Delta\Lambda(\cdot,i)=\Lambda_{1}(\cdot,i)-\Lambda_{2}(\cdot,i),\\
&\Delta p(\cdot,i)=p_{1}(\cdot,i)-p_{2}(\cdot,i),~\Delta\lambda(\cdot,i)=\lambda_{1}(\cdot,i)-\lambda_{2}(\cdot,i),
\end{align*}
and obtain the following BSDE:
\begin{equation}\label{BSDE4}
\left\{
\begin{array}{l}
d\Delta P(t,i)=-[\Pi(t,i,\Delta P,\Delta \Lambda)+H(t,i,P_{1},\Lambda_{1})-H(s,i,P_{2},\Lambda_{2})\\
\qquad\qquad\qquad\qquad+q_{ii}\Delta P(t,i)+\sum_{j\neq i}q_{ij}\Delta p(t,j)]dt+\Delta\Lambda(t,i)dW(t),\\
\Delta P(T,i)=0.
\end{array}
\right.
\end{equation}
Then we focus on
\begin{align*}
&H(t,i,P_{1},\Lambda_{1})-H(s,i,P_{2},\Lambda_{2})\\
&=\left[-\big(\Delta P(t,i)B(t,i)+C(t,i)^{\top}\Delta P(t,i)D(t,i)\big)\big(R(t,i)+D(t,i)^{\top}P_{1}(t,i)D(t,i)\big)^{-1}\big(B(t,i)^{\top}P_{1}(t,i)\right.\\
&\qquad+D^{\top}(t,i)P_{1}(t,i)C(t,i)\big)-\big(P_{2}(t,i)B(t,i)+C(t,i)^{\top}P_{2}(t,i)D(t,i)\big)\big(R(t,i)+D(t,i)^{\top}P_{2}(t,i)D(t,i)\big)^{-1}\\
&\qquad\cdot\big(B(t,i)^{\top}\Delta P(t,i)+D^{\top}(t,i)\Delta P(t,i)C(t,i)\big)+\big(P_{2}(t,i)B(t,i)+C(t,i)^{\top}P_{2}(t,i)D(t,i)\big)\\
&\qquad\cdot\big(R(t,i)+D(t,i)^{\top}P_{2}(t,i)D(t,i)\big)^{-1}D(t,i)^{\top}\Delta P(t,i)D(t,i)\big(R(t,i)+D(t,i)^{\top}P_{1}(t,i)D(t,i)\big)^{-1}\\
&\qquad\left.\cdot\big(B(t,i)^{\top}P_{1}(t,i)+D^{\top}(t,i)P_{1}(t,i)C(t,i)\big)\right]\\
&\quad+\left[-\big(\Delta P(t,i)B(t,i)+C(t,i)^{\top}\Delta P(t,i)D(t,i)\big)\big(R(t,i)+D(t,i)^{\top}P_{1}(t,i)D(t,i)\big)^{-1}D(t,i)^{\top}\Lambda_{1}(t,i)\right.\\
&\quad\qquad-\Lambda_{2}(t,i)D(t,i)\big(R(t,i)+D(t,i)^{\top}P_{2}(t,i)D(t,i)\big)^{-1}\big(B(t,i)^{\top}\Delta P(t,i)+D^{\top}(t,i)\Delta P(t,i)C(t,i)\big)\\
&\quad\qquad+\Lambda_{2}(t,i)D(t,i)\big(R(t,i)+D(t,i)^{\top}P_{2}(t,i)D(t,i)\big)^{-1}D(t,i)^{\top}\Delta P(t,i)D(t,i)\big(R(t,i)\\
&\quad\qquad+D(t,i)^{\top}P_{1}(t,i)D(t,i)\big)^{-1}\big(B(t,i)^{\top}P_{1}(t,i)+D^{\top}(t,i)P_{1}(t,i)C(t,i)\big)+\big(P_{2}(t,i)B(t,i)\\
&\quad\qquad+C(t,i)^{\top}P_{2}(t,i)D(t,i)\big)\big(R(t,i)+D(t,i)^{\top}P_{2}(t,i)D(t,i)\big)^{-1}D(t,i)^{\top}\Delta P(t,i)D(t,i)\\
&\quad\qquad\cdot\big(R(t,i)+D(t,i)^{\top}P_{1}(t,i)D(t,i)\big)^{-1}D(t,i)^{\top}\Lambda_{1}(t,i)+\Lambda_{2}(t,i)D(t,i)\big(R(t,i)+D(t,i)^{\top}P_{2}(t,i)D(t,i)\big)^{-1}\\
&\quad\qquad\left.\cdot D(t,i)^{\top}\Delta P(t,i)D(t,i)\big(R(t,i)+D(t,i)^{\top}P_{1}(t,i)D(t,i)\big)^{-1}D(t,i)^{\top}\Lambda_{1}(t,i)\right]\\
&\quad+\left[-\big(P_{2}(t,i)B(t,i)+C(t,i)^{\top}P_{2}(t,i)D(t,i)\big)\big(R(t,i)+D(t,i)^{\top}P_{2}(t,i)D(t,i)\big)^{-1}D(t,i)^{\top}\Delta\Lambda(t,i)\right.\\
&\quad\qquad\left.-\Delta\Lambda(t,i)D(t,i)\big(R(t,i)+D(t,i)^{\top}P_{1}(t,i)D(t,i)\big)^{-1}\big(B(t,i)^{\top}P_{1}(t,i)+D^{\top}(t,i)P_{1}(t,i)C(t,i)\big)\right]\\
&\quad+\left[-\Lambda_{2}(t,i)D(t,i)\big(R(t,i)+D(t,i)^{\top}P_{2}(t,i)D(t,i)\big)^{-1}D(t,i)^{\top}\Delta\Lambda(t,i)\right.\\
&\quad\qquad\left.-\Delta\Lambda(t,i)D(t,i)\big(R(t,i)+D(t,i)^{\top}P_{1}(t,i)D(t,i)\big)^{-1}D^{\top}(t,i)\Lambda_{1}(t,i)\right]\\
&=:H_{1}(t,i)+H_{2}(t,i)+H_{3}(t,i)+H_{4}(t,i).
\end{align*}
Thanks to $(\mathscr{A}1)$ and $(\mathscr{A}2)$ and $\left(P_{1}(\cdot,i),\Lambda_{1}(\cdot,i)\right)_{i=1}^{l},\left(P_{2}(\cdot,i),\Lambda_{2}(\cdot,i)\right)_{i=1}^{l}\in L_{\mathcal{F}^{W}}^{\infty}(0,T;\mathbb{S}^{n})\times L_{\mathcal{F}^{W}}^{2,\mathrm{bmo}}(0,T;\mathbb{S}^{n\times d})$, $P_{1}(\cdot,i),P_{2}(\cdot,i)\geq0$, there exists nonnegative adapted processes $\bar{\alpha}(\cdot),\bar{\gamma}(\cdot)\in L_{\mathcal{F}^{W}}^{2,\mathrm{bmo}}(0,T;\mathbb{R})$ and $\bar{\beta}(\cdot)\in L_{\mathcal{F}^{W}}^{\infty}(0,T;\mathbb{R})$ such that for $i\in\mathcal{M},t\in[0,T]$,
\begin{align*}
&|\Delta P(t,i)A(t,i)+A(t,i)^{\top}\Delta P(t,i)+C(t,i)^{\top}\Delta P(t,i)C(t,i)+H_{1}(t,i)+H_{2}(t,i)|\leq\bar{\alpha}(t)|\Delta P(t,i)|,\\
&|\Delta\Lambda(t,i)C(t,i)+C(t,i)^{\top}\Delta\Lambda(t,i)+H_{3}(t,i)|\leq\bar{\beta}(t)|\Delta\Lambda(t,i)|,\\
&|H_{4}(t,i)|\leq\bar{\gamma}(t)|\Delta\Lambda(t,i)|.
\end{align*}
Here, we can choose $\bar{\gamma}(t)=2\max\left(|D(t,i)R(t,i)^{-1}D(t,i)^{\top}\Lambda_{1}(t,i)|,|D(t,i)R(t,i)^{-1}D(t,i)^{\top}\Lambda_{2}(t,i)|\right)$.

Applying It\^o's formula to $|\Delta P(t,i)|^{2}$ on $[T-\epsilon,T]$ and taking conditional expectation, we have
\begin{align*}
&|\Delta P(t,i)|^{2}+\mathbb{E}_{t}\left[\int_{t}^{T}|\Delta\Lambda(s,i)|^{2}ds\right]\\
&\quad=2\mathbb{E}_{t}\left[\int_{t}^{T}\langle\Pi(s,i,\Delta P,\Delta\Lambda)+H(s,i,P_{1},\Lambda_{1})-H(s,i,P_{2},\Lambda_{2})\right.\\
&\qquad\qquad\qquad\left.+q_{ii}\Delta P(s,i)+\sum_{j\neq i}q_{ij}\Delta p(s,j),\Delta P(s,i)\rangle ds\right]\\
&\quad\leq\epsilon^{\frac{1}{2}}\mathbb{E}_{t}\left[\int_{t}^{T}\bar{\alpha}(s)^{2}|\Delta P(s,i)|^{2}ds\right]+\epsilon^{-\frac{1}{2}}\mathbb{E}_{t}\left[\int_{t}^{T}|\Delta P(s,i)|^{2}ds\right]+\bar{a}\sum_{j\neq i}\mathbb{E}_{t}\left[\int_{t}^{T}|\Delta p(s,j)|^{2}ds\right]\\
&\qquad+\frac{(l-1)q^{2}}{\bar{a}}\mathbb{E}_{t}\left[\int_{t}^{T}|\Delta P(s,i)|^{2}ds\right]+\bar{b}\mathbb{E}_{t}\left[\int_{t}^{T}|\Delta \Lambda(s,i)|^{2}ds\right]+\frac{1}{\bar{b}}\mathbb{E}_{t}\left[\int_{t}^{T}\bar{\beta}(s)^{2}|\Delta P(s,i)|^{2}ds\right]\\
&\qquad+\bar{c}\mathbb{E}_{t}\left[\int_{t}^{T}|\Delta \Lambda(s,i)|^{2}ds\right]+\frac{1}{\bar{c}}\mathbb{E}_{t}\left[\int_{t}^{T}\bar{\gamma}(s)^{2}|\Delta P(s,i)|^{2}ds\right].
\end{align*}
for some $\epsilon,\bar{a},\bar{b},\bar{c}>0$, which implies
\begin{align*}
&\|\Delta P(\cdot,i)\|^{2}_{L_{\mathcal{F}^{W}}^{\infty}(T-\epsilon,T;\mathbb{S}^{n})}+\|\Delta\Lambda(\cdot,i)\|^{2}_{L_{\mathcal{F}^{W}}^{2,\mathrm{bmo}}(T-\epsilon,T;\mathbb{S}^{n})}\\
&\quad\leq\bar{a}\sum_{j\neq i}\|\Delta p(\cdot,j)\|^{2}_{L_{\mathcal{F}^{W}}^{\infty}(T-\epsilon,T;\mathbb{S}^{n})}+(\bar{b}+\bar{c})\|\Delta\Lambda(\cdot,i)\|^{2}_{L_{\mathcal{F}^{W}}^{2,\mathrm{bmo}}(T-\epsilon,T;\mathbb{S}^{n})}\\
&\qquad+\left(\epsilon^{\frac{1}{2}}\left\|\bar{\alpha}(\cdot)\right\|^{2}_{L_{\mathcal{F}^{W}}^{2,\mathrm{bmo}}(0,T;\mathbb{R}^{n\times n})}+\epsilon^{\frac{1}{2}}+\frac{(l-1)q^{2}\epsilon}{\bar{a}}+\frac{\epsilon}{\bar{b}}\left\|\bar{\beta}(\cdot)\right\|^{2}_{L_{\mathcal{F}^{W}}^{\infty}(0,T;\mathbb{R}^{n\times n})}\right.\\
&\quad\qquad\left.+\frac{1}{\bar{c}}\left\|\bar{\gamma}(\cdot)\right\|^{2}_{L_{\mathcal{F}^{W}}^{2,\mathrm{bmo}}(0,T;\mathbb{R}^{n\times n})}\right)\|\Delta P(\cdot,i)\|^{2}_{L_{\mathcal{F}^{W}}^{\infty}(T-\epsilon,T;\mathbb{S}^{n})}.
\end{align*}
From \eqref{condition}, when $L_{\sigma}>0$ is sufficiently small, we have
\begin{align*}
\|\bar{\gamma}(\cdot)\|_{L_{\mathcal{F}^{W}}^{2,\mathrm{bmo}}(0,T;\mathbb{R}^{n\times n})}<1,
\end{align*}
then there exists suitable $\epsilon,\bar{a},\bar{b},\bar{c}>0$ such that
\begin{equation}\label{epsilon}
\begin{aligned}
&1-\epsilon^{\frac{1}{2}}\left\|\bar{\alpha}(\cdot)\right\|^{2}_{L_{\mathcal{F}^{W}}^{2,\mathrm{bmo}}(0,T;\mathbb{R}^{n\times n})}-\epsilon^{\frac{1}{2}}-\frac{(l-1)q^{2}\epsilon}{\bar{a}}-\frac{\epsilon}{\bar{b}}\left\|\bar{\beta}(\cdot)\right\|^{2}_{L_{\mathcal{F}^{W}}^{\infty}(0,T;\mathbb{R}^{n\times n})}-\frac{1}{\bar{c}}\left\|\bar{\gamma}(\cdot)\right\|^{2}_{L_{\mathcal{F}^{W}}^{2,\mathrm{bmo}}(0,T;\mathbb{R}^{n\times n})}>0,\\
&1-\bar{b}-\bar{c}>0,\\
&\frac{\bar{a}(l-1)}{1-\epsilon^{\frac{1}{2}}\left\|\bar{\alpha}(\cdot)\right\|^{2}_{L_{\mathcal{F}^{W}}^{2,\mathrm{bmo}}(0,T;\mathbb{R}^{n\times n})}-\epsilon^{\frac{1}{2}}-\frac{(l-1)q^{2}\epsilon}{\bar{a}}-\frac{\epsilon}{\bar{b}}\left\|\bar{\beta}(\cdot)\right\|^{2}_{L_{\mathcal{F}^{W}}^{\infty}(0,T;\mathbb{R}^{n\times n})}-\frac{1}{\bar{c}}\left\|\bar{\gamma}(\cdot)\right\|^{2}_{L_{\mathcal{F}^{W}}^{2,\mathrm{bmo}}(0,T;\mathbb{R}^{n\times n})}}<1.
\end{aligned}
\end{equation}
Let
\begin{align*}
\bar{d}=\frac{\bar{a}(l-1)}{1-\epsilon^{\frac{1}{2}}\left\|\bar{\alpha}(\cdot)\right\|^{2}_{L_{\mathcal{F}^{W}}^{2,\mathrm{bmo}}(0,T;\mathbb{R}^{n\times n})}-\epsilon^{\frac{1}{2}}-\frac{(l-1)q^{2}\epsilon}{\bar{a}}-\frac{\epsilon}{\bar{b}}\left\|\bar{\beta}(\cdot)\right\|^{2}_{L_{\mathcal{F}^{W}}^{\infty}(0,T;\mathbb{R}^{n\times n})}-\frac{1}{\bar{c}}\left\|\bar{\gamma}(\cdot)\right\|^{2}_{L_{\mathcal{F}^{W}}^{2,\mathrm{bmo}}(0,T;\mathbb{R}^{n\times n})}}
\end{align*}
and we can obtain
\begin{align*}
&\frac{a(l-1)}{\bar{d}}\sum_{i=1}^{l}\|\Delta P(\cdot,i)\|^{2}_{L_{\mathcal{F}^{W}}^{\infty}(T-\epsilon,T;\mathbb{S}^{n})}+(1-\bar{b}-\bar{c})\sum_{i=1}^{l}\|\Delta\Lambda(\cdot,i)\|^{2}_{L_{\mathcal{F}^{W}}^{2,\mathrm{bmo}}(T-\epsilon,T;\mathbb{S}^{n})}\\
&\leq a(l-1)\sum_{i=1}^{l}\|\Delta p(\cdot,i)\|^{2}_{L_{\mathcal{F}^{W}}^{\infty}(T-\epsilon,T;\mathbb{S}^{n})},
\end{align*}
which means SRE \eqref{Riccati} has a unique solution $\left(P(\cdot,i),\Lambda(\cdot,i)\right)_{i=1}^{l}$ on $[T-\epsilon,T]$ such that $\left(P(\cdot,i),\Lambda(\cdot,i)\right)\in L_{\mathcal{F}^{W}}^{\infty}(T-\epsilon,T;\mathbb{S}^{n})\times L_{\mathcal{F}^{W}}^{2,\mathrm{bmo}}(T-\epsilon,T;\mathbb{S}^{n})$. With this method we can solve SRE \eqref{Riccati} on $[T-2\epsilon,T-\epsilon],[T-3\epsilon,T-2\epsilon],\ldots,$ and finally $[0,T]$.
\end{proof}

\subsection{Solvability of adjoint BSDE \eqref{K}}
On the basis of SRE \eqref{Riccati}, we study the solvability for BSDE \eqref{K}, which is the coupled system with unbounded coefficients. At first, we introduce the following lemma, which is useful in the subsequent proof.

\begin{lemma}\label{le-b}
The following BSDE
\begin{equation}\label{BSDE}
\begin{aligned}
\left\{\begin{array}{l}
dY(t)=-[\alpha(t)^{\top}Y(t)+\beta(t)^{\top}Z(t)+\gamma(t)^{\top}Z(t)+\eta(t)]dt+Z(t)dW(t),~t\in[0,T],\\
Y(T)=\xi,
\end{array}\right.
\end{aligned}
\end{equation}
with $\alpha(\cdot),\gamma(\cdot)\in L_{\mathcal{F}^{W}}^{2,\mathrm{bmo}}(0,T;\mathbb{R}^{k\times k})$, $\eta(\cdot)\in L_{\mathcal{F}^{W}}^{2,\mathrm{bmo}}(0,T;\mathbb{R}^{k})$, $\beta(\cdot)\in L_{\mathcal{F}^{W}}^{\infty}(0,T;\mathbb{R}^{k\times k})$ and\\
$\|\gamma(\cdot)\|_{L_{\mathcal{F}^{W}}^{2,\mathrm{BMO}}(0,T;\mathbb{R}^{n\times n})}<1$, admits a unique solution $\left(Y(\cdot),Z(\cdot)\right)$ such that $\left(Y(\cdot),Z(\cdot)\right)\in L_{\mathcal{F}^{W}}^{\infty}(0,T;\mathbb{R}^{k})\times L_{\mathcal{F}^{W}}^{2,\mathrm{bmo}}(0,T;\mathbb{R}^{k})$.
\end{lemma}
\begin{proof}
We denote $\left(y(\cdot),z(\cdot)\right)\in L_{\mathcal{F}^{W}}^{\infty}(0,T;\mathbb{R}^{k})\times L_{\mathcal{F}^{W}}^{2,\mathrm{bmo}}(0,T;\mathbb{R}^{k})$
and consider the following BSDEs:
\begin{equation}\label{BSDE2}
\begin{aligned}
\left\{\begin{array}{l}
dY(t)=-[\alpha(t)^{\top}y(t)+\beta(t)^{\top}z(t)+\gamma(t)^{\top}z(t)+\eta(t)]dt+Z(t)dW(t),~t\in[0,T],\\
Y(T)=\xi.
\end{array}\right.
\end{aligned}
\end{equation}
From \cite[Proposition 2.1]{Sun and Yong 2014}, BSDE \eqref{BSDE2} has a unique solution $\left(y(\cdot),z(\cdot)\right)\in L_{\mathcal{F}^{W}}^{p}(\Omega;C(0,T;\mathbb{R}^{k}))\times L_{\mathcal{F}^{W}}^{1}(0,T;\mathbb{R}^{k})$. 

At first, we are going to prove for some constants $\varepsilon,a,d>0$, $\Gamma(y(\cdot),z(\cdot)):=(Y(\cdot),Z(\cdot))$ is a contraction map on the closed convex set $\mathscr{B}_{\varepsilon}$ by
\begin{equation*}
\begin{aligned}
\mathscr{B}_{\varepsilon}:=&\bigg\{(y(\cdot),z(\cdot))\in L_{\mathcal{F}^{W}}^{\infty}(T-\varepsilon,T;\mathbb{R}^{k})\times L_{\mathcal{F}^{W}}^{2,\mathrm{bmo}}(T-\varepsilon,T;\mathbb{R}^{k}): \frac{a}{d}\|y(\cdot)\|^{2}_{L_{\mathcal{F}^{W}}^{\infty}(T-\varepsilon,T;\mathbb{R}^{k})}\\
&\qquad+\|z(\cdot)\|^{2}_{L_{\mathcal{F}^{W}}^{2,\mathrm{bmo}}(T-\varepsilon,T;\mathbb{R}^{k})}\leq\frac{1}{1-d}\left[\|\xi\|^{2}_{L_{\mathcal{F}^{W}}^{\infty}(\Omega;\mathbb{R}^{k})}+\|\eta(\cdot)\|^{2}_{L_{\mathcal{F}^{W}}^{2,\mathrm{bmo}}(T-\varepsilon,T;\mathbb{R}^{k})}\right]\bigg\},
\end{aligned}
\end{equation*}
where $(y(\cdot),z(\cdot))$ is defined on $\Omega\times[T-\varepsilon,T]$ and $\varepsilon,a,d$ will be determined later (see \eqref{varepsilon}). Applying It\^o's formula to $|Y(t)|^{2}$ on $[T-\varepsilon,T]$ and taking conditional expectation, we have
\begin{align*}
&|Y(t)|^{2}+\mathbb{E}_{t}\left[\int_{t}^{T}|Z(s)|^{2}ds\right]\\
&=\mathbb{E}_{t}\left[|\xi|^{2}\right]+2\mathbb{E}_{t}\left[\int_{t}^{T}\langle\alpha(s)^{\top}y(s)+\beta(s)^{\top}z(s)+\gamma(s)^{\top}z(s)+\eta(s),Y(s)\rangle ds\right]\\
&\leq\mathbb{E}_{t}\left[|\xi|^{2}\right]+\frac{a}{\varepsilon}\mathbb{E}_{t}\left[\int_{t}^{T}|y(s)|^{2}ds\right]+\frac{\varepsilon}{a}\mathbb{E}_{t}\left[\int_{t}^{T}|\alpha(s)Y(s)|^{2}ds\right]+b\mathbb{E}_{t}\left[\int_{t}^{T}|z(s)|^{2}ds\right]\\
&\quad+\frac{1}{b}\mathbb{E}_{t}\left[\int_{t}^{T}|\beta(s)Y(s)|^{2}ds\right]+c\mathbb{E}_{t}\left[\int_{t}^{T}|z(s)|^{2}ds\right]+\frac{1}{c}\mathbb{E}_{t}\left[\int_{t}^{T}|\gamma(s)Y(s)|^{2}ds\right]\\
&\quad+\mathbb{E}_{t}\left[\int_{t}^{T}|\eta(s)|^{2}ds\right]+\mathbb{E}_{t}\left[\int_{t}^{T}|Y(s)|^{2}ds\right]\\
&\leq\|\xi\|^{2}_{L_{\mathcal{F}^{W}}^{\infty}(\Omega;\mathbb{R}^{k})}+\|\eta(\cdot)\|^{2}_{L_{\mathcal{F}^{W}}^{2,\mathrm{bmo}}(T-\varepsilon,T;\mathbb{R}^{k})}+a\|y(\cdot)\|^{2}_{L_{\mathcal{F}^{W}}^{\infty}(T-\varepsilon,T;\mathbb{R}^{k})}+(b+c)\|z(\cdot)\|^{2}_{L_{\mathcal{F}^{W}}^{2,\mathrm{bmo}}(T-\varepsilon,T;\mathbb{R}^{k})}\\
&\quad+\left(\frac{\varepsilon}{a}\left\|\alpha(\cdot)\right\|^{2}_{L_{\mathcal{F}^{W}}^{2,\mathrm{bmo}}(0,T;\mathbb{R}^{k\times k})}+\frac{\varepsilon}{b}\left\|\beta(\cdot)\right\|^{2}_{L_{\mathcal{F}^{W}}^{\infty}(0,T;\mathbb{R}^{k\times k})}+\frac{1}{c}\left\|\gamma(\cdot)\right\|^{2}_{L_{\mathcal{F}^{W}}^{2,\mathrm{bmo}}(0,T;\mathbb{R}^{k\times k})}+\varepsilon\right)\|Y(\cdot)\|^{2}_{L_{\mathcal{F}^{W}}^{\infty}(T-\varepsilon,T;\mathbb{R}^{k})}
\end{align*}
for some $\varepsilon,a,b,c>0$, which implies
\begin{equation}\label{bounds}
\begin{aligned}
&\|Y(\cdot)\|^{2}_{L_{\mathcal{F}^{W}}^{\infty}(T-\varepsilon,T;\mathbb{R}^{k})}+\|Z(\cdot)\|^{2}_{L_{\mathcal{F}^{W}}^{2,\mathrm{bmo}}(T-\varepsilon,T;\mathbb{R}^{k})}\\
&\leq\|\xi\|^{2}_{L_{\mathcal{F}^{W}}^{\infty}(\Omega;\mathbb{R}^{k})}+\|\eta(\cdot)\|^{2}_{L_{\mathcal{F}^{W}}^{2,\mathrm{bmo}}(T-\varepsilon,T;\mathbb{R}^{k})}+a\|y(\cdot)\|^{2}_{L_{\mathcal{F}^{W}}^{\infty}(T-\varepsilon,T;\mathbb{R}^{k})}+(b+c)\|z(\cdot)\|^{2}_{L_{\mathcal{F}^{W}}^{2,\mathrm{bmo}}(T-\varepsilon,T;\mathbb{R}^{k})}\\
&\quad+\left(\frac{\varepsilon}{a}\left\|\alpha(\cdot)\right\|^{2}_{L_{\mathcal{F}^{W}}^{2,\mathrm{bmo}}(0,T;\mathbb{R}^{k\times k})}+\frac{\varepsilon}{b}\left\|\beta(\cdot)\right\|^{2}_{L_{\mathcal{F}^{W}}^{\infty}(0,T;\mathbb{R}^{k\times k})}+\frac{1}{c}\left\|\gamma(\cdot)\right\|^{2}_{L_{\mathcal{F}^{W}}^{2,\mathrm{bmo}}(0,T;\mathbb{R}^{k\times k})}+\varepsilon\right)\|Y(\cdot)\|^{2}_{L_{\mathcal{F}^{W}}^{\infty}(T-\varepsilon,T;\mathbb{R}^{k})}.
\end{aligned}
\end{equation}
Since
\begin{align*}
\|\gamma(\cdot)\|_{L_{\mathcal{F}^{W}}^{2,\mathrm{bmo}}(0,T;\mathbb{R}^{k\times k})}<1,
\end{align*}
then there exists suitable $\varepsilon,a,b,c>0$ such that
\begin{equation}\label{varepsilon}
\begin{aligned}
&1-\frac{\varepsilon}{a}\left\|\alpha(\cdot)\right\|^{2}_{L_{\mathcal{F}^{W}}^{2,\mathrm{bmo}}(0,T;\mathbb{R}^{k\times k})}-\frac{\varepsilon}{b}\left\|\beta(\cdot)\right\|^{2}_{L_{\mathcal{F}^{W}}^{\infty}(0,T;\mathbb{R}^{k\times k})}-\frac{1}{c}\left\|\gamma(\cdot)\right\|^{2}_{L_{\mathcal{F}^{W}}^{2,\mathrm{bmo}}(0,T;\mathbb{R}^{k\times k})}-\varepsilon>0,\\
&\frac{a}{1-\frac{\varepsilon}{a}\left\|\alpha(\cdot)\right\|^{2}_{L_{\mathcal{F}^{W}}^{2,\mathrm{bmo}}(0,T;\mathbb{R}^{k\times k})}-\frac{\varepsilon}{b}\left\|\beta(\cdot)\right\|^{2}_{L_{\mathcal{F}^{W}}^{\infty}(0,T;\mathbb{R}^{k\times k})}-\frac{1}{c}\left\|\gamma(\cdot)\right\|^{2}_{L_{\mathcal{F}^{W}}^{2,\mathrm{bmo}}(0,T;\mathbb{R}^{k\times k})}-\varepsilon}<1,\\
&b+c<1.
\end{aligned}
\end{equation}
Therefore we denote
\begin{align*}
d=\max\left(\frac{a}{1-\frac{\varepsilon}{a}\left\|\alpha(\cdot)\right\|^{2}_{L_{\mathcal{F}^{W}}^{2,\mathrm{bmo}}(0,T;\mathbb{R}^{k\times k})}-\frac{\varepsilon}{b}\left\|\beta(\cdot)\right\|^{2}_{L_{\mathcal{F}^{W}}^{\infty}(0,T;\mathbb{R}^{k\times k})}-\frac{1}{c}\left\|\gamma(\cdot)\right\|^{2}_{L_{\mathcal{F}^{W}}^{2,\mathrm{bmo}}(0,T;\mathbb{R}^{k\times k})}-\varepsilon},b+c\right)
\end{align*}
and obtain
\begin{align*}
&\frac{a}{d}\|Y(\cdot)\|^{2}_{L_{\mathcal{F}^{W}}^{\infty}(T-\varepsilon,T;\mathbb{R}^{k})}+\|Z(\cdot)\|^{2}_{L_{\mathcal{F}^{W}}^{2,\mathrm{bmo}}(T-\varepsilon,T;\mathbb{R}^{k})}\\
&\leq d\left[\frac{a}{d}\|y(\cdot)\|^{2}_{L_{\mathcal{F}^{W}}^{\infty}(T-\varepsilon,T;\mathbb{R}^{k})}+\|z(\cdot)\|^{2}_{L_{\mathcal{F}^{W}}^{2,\mathrm{bmo}}(T-\varepsilon,T;\mathbb{R}^{k})}\right]+\|\xi\|^{2}_{L_{\mathcal{F}^{W}}^{\infty}(\Omega;\mathbb{R}^{k})}+\|\eta(\cdot)\|^{2}_{L_{\mathcal{F}^{W}}^{2,\mathrm{bmo}}(T-\varepsilon,T;\mathbb{R}^{k})}\\
&\leq\frac{1}{1-d}\left[\|\xi\|^{2}_{L_{\mathcal{F}^{W}}^{\infty}(\Omega;\mathbb{R}^{k})}+\|\eta(\cdot)\|^{2}_{L_{\mathcal{F}^{W}}^{2,\mathrm{bmo}}(T-\varepsilon,T;\mathbb{R}^{k})}\right],
\end{align*}
yielding $(y(\cdot),z(\cdot))\in\mathscr{B}_{\varepsilon}$. For $(y_{1}(\cdot),z_{1}(\cdot)),(y_{2}(\cdot),z_{2}(\cdot))\in\mathscr{B}_{\varepsilon}$, we denote
\begin{align*}
(Y_{1}(\cdot),Z_{1}(\cdot))=\Gamma(y_{1}(\cdot),z_{1}(\cdot)),~(Y_{2}(\cdot),Z_{2}(\cdot))=\Gamma(y_{2}(\cdot),z_{2}(\cdot))
\end{align*}
and 
\begin{align*}
&\Delta Y(\cdot)=Y_{1}(\cdot)-Y_{2}(\cdot),~\Delta Z(\cdot)=Z_{1}(\cdot)-Z_{2}(\cdot),\\
&\Delta y(\cdot)=y_{1}(\cdot)-y_{2}(\cdot),~\Delta z(\cdot)=z_{1}(\cdot)-z_{2}(\cdot).
\end{align*}
Similar to \eqref{bounds}, we can also obtain
\begin{align*}
&\|\Delta Y(\cdot)\|^{2}_{L_{\mathcal{F}^{W}}^{\infty}(T-\varepsilon,T;\mathbb{R}^{k})}+\|\Delta Z(\cdot)\|^{2}_{L_{\mathcal{F}^{W}}^{2,\mathrm{bmo}}(T-\varepsilon,T;\mathbb{R}^{k})}\\
&\leq a\|\Delta y(\cdot)\|^{2}_{L_{\mathcal{F}^{W}}^{\infty}(T-\varepsilon,T;\mathbb{R}^{k})}+(b+c)\|\Delta z(\cdot)\|^{2}_{L_{\mathcal{F}^{W}}^{2,\mathrm{bmo}}(T-\varepsilon,T;\mathbb{R}^{k})}\\
&\quad+\left(\frac{\varepsilon}{a}\left\|\alpha(\cdot)\right\|^{2}_{L_{\mathcal{F}^{W}}^{2,\mathrm{bmo}}(0,T;\mathbb{R}^{k\times k})}+\frac{\varepsilon}{b}\left\|\beta(\cdot)\right\|^{2}_{L_{\mathcal{F}^{W}}^{\infty}(0,T;\mathbb{R}^{k\times k})}+\frac{1}{c}\left\|\gamma(\cdot)\right\|^{2}_{L_{\mathcal{F}^{W}}^{2,\mathrm{bmo}}(0,T;\mathbb{R}^{k\times k})}\right)\|\Delta Y(\cdot)\|^{2}_{L_{\mathcal{F}^{W}}^{\infty}(T-\varepsilon,T;\mathbb{R}^{k})}.
\end{align*}
and thus
\begin{align*}
&\frac{a}{d}\|Y(\cdot)\|^{2}_{L_{\mathcal{F}^{W}}^{\infty}(T-\varepsilon,T;\mathbb{R}^{k})}+\|Z(\cdot)\|^{2}_{L_{\mathcal{F}^{W}}^{2,\mathrm{bmo}}(T-\varepsilon,T;\mathbb{R}^{k})}\\
&\quad\leq d\left[\frac{a}{d}\|y(\cdot)\|^{2}_{L_{\mathcal{F}^{W}}^{\infty}(T-\varepsilon,T;\mathbb{R}^{k})}+\|z(\cdot)\|^{2}_{L_{\mathcal{F}^{W}}^{2,\mathrm{bmo}}(T-\varepsilon,T;\mathbb{R}^{k})}\right].
\end{align*}
Now we have proven that BSDE \eqref{BSDE} has a unique solution $\left(Y(\cdot),Z(\cdot)\right)$ on $[T-\varepsilon,T]$ such that $\left(Y(\cdot),Z(\cdot)\right)\in L_{\mathcal{F}^{W}}^{\infty}(T-\varepsilon,T;\mathbb{R}^{k})\times L_{\mathcal{F}^{W}}^{2,\mathrm{bmo}}(T-\varepsilon,T;\mathbb{R}^{k})$. Following the same method, we can prove that BSDE \eqref{BSDE} has a unique solution $\left(Y(\cdot),Z(\cdot)\right)$ on $[T-2\varepsilon,T-\varepsilon]$ such that $\left(Y(\cdot),Z(\cdot)\right)\in L_{\mathcal{F}^{W}}^{\infty}(T-2\varepsilon,T-\varepsilon;\mathbb{R}^{k})\times L_{\mathcal{F}^{W}}^{2,\mathrm{bmo}}(T-2\varepsilon,T-\varepsilon;\mathbb{R}^{k})$. To conclude, we can prove the solvability of BSDE \eqref{BSDE} on [0,T].
\end{proof}

Now we are prepared to prove our second main result.

\begin{proof}[Proof\ of\ Theorem \ref{th-2}]
For each $i\in\mathcal{M}$ and any $t\in[0,T]$, we set
\begin{align*}
&\alpha(t,i)=A(t,i)-B(t,i)\Gamma(t,i),\\
&\beta(t,i)=C(t,i)-D(t,i)\left(R(t,i)+D(t,i)^{\top}P(t,i)D(t,i)\right)^{-1}\left(B(t,i)^{\top}P(t,i)+D(t,i)^{\top}P(t,i)C(t,i)\right),\\
&\gamma(t,i)=-D(t,i)\left(R(t,i)+D(t,i)^{\top}P(t,i)D(t,i)\right)^{-1}D(t,i)^{\top}\Lambda(t,i),\\
&\eta(t,i)=\Gamma(t,i)^{\top}\left(D(t,i)^{\top}P(t,i)\sigma(t,i)-R(t,i)r(t,i)\right)+Q(t,i)q(t,i)-P(t,i)b(t,i),\\
&\qquad\qquad-C(t,i)^{\top}P(t,i)\sigma(t,i)+\Lambda(t,i)\sigma(t,i),\\
&\xi(i)=G(i)g(i).
\end{align*}
Then we can rewrite BSDE \eqref{K} as
\begin{equation}\label{K'}
\left\{
\begin{array}{l}
\begin{aligned}
dK(t,i)=&-\left[\alpha(t,i)^{\top}K(t,i)+\beta(t,i)^{\top}L(t,i)+\gamma(t,i)^{\top}L(t,i)+\eta(t,i)+\textstyle\sum_{j=1}^{l}q_{ij}K(t,j)\right]dt\\
&+L(t,i)dW_{t},\quad t\in[0,T]
\end{aligned}\\
K(T,i)=\xi(i),\quad i\in\mathcal{M}.
\end{array}
\right.
\end{equation}
We define
\begin{align*}
&\underline{\alpha}(t)=
\begin{bmatrix}
    \alpha(t,1)            &           &\multicolumn{2}{c}{\text{\large 0}}\\
		                   &\alpha(t,2)&      &                                   \\
		                   &           &\ddots& \\
	\multicolumn{2}{c}{\text{\large 0}}&      &\alpha(t,l)
\end{bmatrix}
+
\begin{bmatrix}
	q_{11}I_{n}&q_{12}I_{n}&\cdots&q_{1l}I_{n}\\
	q_{21}I_{n}&q_{22}I_{n}&\cdots&q_{2l}I_{n}\\
	\vdots	   &\vdots     &\ddots&\vdots\\
	q_{l1}I_{n}&q_{l2}I_{n}&\cdots&q_{ll}I_{n}
\end{bmatrix}^{\top},\\
&\underline{\beta}(t)=
\begin{bmatrix}
    \beta(t,1)            &            &\multicolumn{2}{c}{\text{\large 0}}\\
		                  &\beta(t,2)  &      &                                   \\
		                  &            &\ddots& \\
	\multicolumn{2}{c}{\text{\large 0}}&      &\beta(t,l)
\end{bmatrix},\quad
\underline{\gamma}(t)=
\begin{bmatrix}
    \gamma(t,1)            &           &\multicolumn{2}{c}{\text{\large 0}}\\
		                   &\gamma(t,2)&      &                                   \\
		                   &           &\ddots& \\
	\multicolumn{2}{c}{\text{\large 0}}&      &\gamma(t,l)
\end{bmatrix},\\
&K(t)=
\begin{bmatrix}
    K(t,1)\\
    K(t,2)\\
    \vdots\\
    K(t,l)\\
\end{bmatrix},\quad
L(t)=
\begin{bmatrix}
    L(t,1)\\
    L(t,2)\\
    \vdots\\
    L(t,l)\\
\end{bmatrix},\quad
\underline{\eta}(t)=
\begin{bmatrix}
    \eta(t,1)\\
    \eta(t,2)\\
    \vdots\\
    \eta(t,l)\\
\end{bmatrix},\quad
\underline{\xi}=
\begin{bmatrix}
    \xi(1)\\
    \xi(2)\\
    \vdots\\
    \xi(l)\\
\end{bmatrix}
\end{align*}
and again rewrite BSDE \eqref{K'} as
\begin{equation}\label{K''}
\left\{
\begin{array}{l}
dK(t)=-\big[\underline{\alpha}(t)^{\top}K(t)+\underline{\beta}(t)^{\top}L(t)+\underline{\gamma}(t)^{\top}L(t)+\underline{\eta}(t)\big]dt+L(t)dW_{t},\quad t\in[0,T],\\
K(T)=\underline{\xi}.
\end{array}
\right.
\end{equation}
From \eqref{condition}, when $L_{\sigma}>0$ is sufficiently small, we have
\begin{align*}
\|\underline{\gamma}(\cdot)\|_{L_{\mathcal{F}^{W}}^{2,\mathrm{bmo}}(0,T;\mathbb{R}^{nl\times nl})}<1.
\end{align*}
From Lemma \ref{le-b}, BSDE \eqref{K''} has a unique solution $\left(K(\cdot),L(\cdot)\right)$ such that $\left(K(\cdot),L(\cdot)\right)\in L_{\mathcal{F}^{W}}^{\infty}(0,T;\mathbb{R}^{nl})\times L_{\mathcal{F}^{W}}^{2,\mathrm{bmo}}(0,T;\mathbb{R}^{nl})$, which finishes the proof.
\end{proof}

\subsection{Optimal control and optimal value}
Thanks to the solvability of SRE \eqref{Riccati} and BSDE \eqref{K}, we can now represent the optimal control and optimal value with their unique solutions and give the following proof.

\begin{proof}[Proof\ of\ Theorem \ref{th-3}]
From the proof of Lemma 4.2 in \cite{Hu et al. 2022b},  $u^{*}$ is clearly an admissible control. Without causing ambiguity, we will abbreviate $\psi(t,\alpha_{t})$ to $\psi$, $\psi=A,B,C,D,b,\sigma,Q,R,q,p,P,\Lambda,K,L,X,u$. Applying It\^o's formula to $t\mapsto \langle P(t,\alpha_{t})X(t),X(t))\rangle$ and $t\mapsto -2\langle K(t,\alpha_{t}),X(t))\rangle$, and taking the expectation, it holds that
\begin{equation}\label{PXX}
\begin{aligned}
&\mathbb{E}\left[\left\langle G(\alpha_{T})X(T),X(T)\right\rangle\right]-\left\langle P(0,i_{0})x,x\right\rangle\\
&\quad=2\mathbb{E}\left[\int_{0}^{T}\langle PBu+Pb+\Lambda Du+\Lambda \sigma+C^{\top}PDu+C^{\top}P\sigma,X\rangle dt\right]\\
&\qquad+\mathbb{E}\left[\int_{0}^{T}\langle(R+D^{\top}PD)^{-1}(B^{\top}P+D^{\top}PC+D^{\top}\Lambda)X,(B^{\top}P+D^{\top}PC+D^{\top}\Lambda)X\rangle dt\right]\\
&\qquad+\mathbb{E}\left[\int_{0}^{T}\langle P(Du+\sigma),Du+\sigma\rangle dt\right]-\mathbb{E}\left[\int_{0}^{T}\langle QX,X\rangle dt\right]
\end{aligned}
\end{equation}
and
\begin{equation}\label{KX}
\begin{aligned}
&-2\mathbb{E}\left[\left\langle G(\alpha_{T})g(\alpha_{T}),X(T)\right\rangle\right]+2\left\langle K(0,i_{0}),x\right\rangle\\
&\quad=-2\mathbb{E}\left[\int_{0}^{T}\langle(R+D^{\top}PD)^{-1}(B^{\top}P+D^{\top}PC+D^{\top}\Lambda)X,B^{\top}K+D^{\top}L-D^{\top}P\sigma+Rr\rangle dt\right]\\
&\qquad+2\mathbb{E}\left[\int_{0}^{T}\langle Qq-Pb-C^{\top}P\sigma-\Lambda\sigma,X\rangle dt\right]-2\mathbb{E}\left[\int_{0}^{T}\langle Bu+b,K\rangle dt\right]\\
&\qquad-2\mathbb{E}\left[\int_{0}^{T}\langle Du+\sigma,L\rangle dt\right].
\end{aligned}
\end{equation}
Now we denote
\begin{align*}
v(t,i)=&-\left(R(t,i)+D(t,i)P(t,i)D(t,i)^{\top}\right)^{-1}\\
&\cdot\left[\left(B(t,i)^{\top}P(t,i)+D(t,i)^{\top}P(t,i)C(t,i)+D(t,i)^{\top}\Lambda(t,i)\right)X\right.\\
&\quad+\left.D(t,i)^{\top}P(t,i)\sigma(t,i)-R(t,i)r(t,i)-B(t,i)^{\top}K(t,i)-D(t,i)^{\top}L(t,i)\right]
\end{align*}
and add \eqref{PXX} and \eqref{KX}, then we can obtain
\begin{align*}
&\mathbb{E}\left[\left\langle G(\alpha_{T})X(T),X(T)\right\rangle-2\left\langle G(\alpha_{T})g(\alpha_{T}),X(T)\right\rangle\right]-\left\langle P(0,i_{0})x,x\right\rangle+2\left\langle K(0,i_{0}),x\right\rangle\\
&\quad=2\mathbb{E}\left[\int_{0}^{T}\langle PBu+Pb+\Lambda Du+\Lambda \sigma+C^{\top}PDu+C^{\top}P\sigma+Qq-Pb-C^{\top}P\sigma-\Lambda\sigma,X\rangle dt\right]\\
&\qquad+\mathbb{E}\left[\int_{0}^{T}\langle(R+D^{\top}PD)^{-1}(B^{\top}P+D^{\top}PC+D^{\top}\Lambda)X,(B^{\top}P+D^{\top}PC+D^{\top}\Lambda)X\rangle dt\right]\\
&\qquad-2\mathbb{E}\left[\int_{0}^{T}\langle(R+D^{\top}PD)^{-1}(B^{\top}P+D^{\top}PC+D^{\top}\Lambda)X,B^{\top}K+D^{\top}L-D^{\top}P\sigma+Rr\rangle dt\right]\\
&\qquad+\mathbb{E}\left[\int_{0}^{T}\langle P(Du+\sigma)-2L,Du+\sigma\rangle dt\right]-\mathbb{E}\left[\int_{0}^{T}\langle QX,X\rangle dt\right]-2\mathbb{E}\left[\int_{0}^{T}\langle Bu+b,K\rangle dt\right]\\
&\quad=\mathbb{E}\left[\int_{0}^{T}\langle (R+D^{\top}PD)(u-v),u-v\rangle dt\right]-\mathbb{E}\left[\int_{0}^{T}\langle Ru,u-2r\rangle dt\right]-\mathbb{E}\left[\int_{0}^{T}\langle QX,X-2q\rangle dt\right]\\
&\qquad-\mathbb{E}\left[\int_{0}^{T}\langle (R+D^{\top}PD)^{-1}(D^{\top}P\sigma-Rr-B^{\top}K-D^{\top}L),D^{\top}P\sigma-Rr-B^{\top}K-D^{\top}L\rangle dt\right]\\
&\qquad+\mathbb{E}\left[\int_{0}^{T}\langle P\sigma,\sigma\rangle dt\right]-2\mathbb{E}\left[\int_{0}^{T}\langle b,K\rangle dt\right]-2\mathbb{E}\left[\int_{0}^{T}\langle\sigma,L\rangle dt\right],
\end{align*}
which implies
\begin{align*}
&\mathbb{E}\left[\left\langle G(\alpha_{T})(X(T)-g(\alpha_{T})),X(T)-g(\alpha_{T})\right\rangle\right]+\mathbb{E}\left[\int_{0}^{T}\big(\langle Q(X-q), X-q\rangle+\langle R(u-r), u-r\rangle\big)dt\right]\\
&\quad=\langle P(0,i_{0})x,x\rangle-2\langle K(0,i_{0}),x\rangle+\mathbb{E}[\langle G(\alpha_{T})g(\alpha_{T}),g(\alpha_{T})\rangle]\\
&\qquad+\mathbb{E}\left[\int_{0}^{T}\langle (R+D^{\top}PD)^{-1}(u-v),u-v\rangle dt\right]+\mathbb{E}\left[\int_{0}^{T}\langle Qq,q\rangle dt\right]+\mathbb{E}\left[\int_{0}^{T}\langle Rr,r\rangle dt\right]\\
&\qquad-\mathbb{E}\left[\int_{0}^{T}\langle (R+D^{\top}PD)^{-1}(D^{\top}P\sigma-Rr-B^{\top}K-D^{\top}L),D^{\top}P\sigma-Rr-B^{\top}K-D^{\top}L\rangle dt\right]\\
&\qquad+\mathbb{E}\left[\int_{0}^{T}\langle P\sigma,\sigma\rangle dt\right]-2\mathbb{E}\left[\int_{0}^{T}\langle b,K\rangle dt\right]-2\mathbb{E}\left[\int_{0}^{T}\langle\sigma,L\rangle dt\right].
\end{align*}
Since for any $t\in[0,T],~i\in\mathcal{M}$, we have
\begin{align*}
R(t,i)+D(t,i)^{\top}P(t,i)D(t,i)>0,
\end{align*}
then it holds that
\begin{align*}
&\mathbb{E}\left[\left\langle G(\alpha_{T})(X(T)-g(\alpha_{T})),X(T)-g(\alpha_{T})\right\rangle\right]+E\left[\int_{0}^{T}\big(\langle Q(X-q), X-q\rangle+\langle R(u-p), u-p\rangle\big)dt\right]\\
&\quad\geq\langle P(0,i_{0})x,x\rangle-2\langle K(0,i_{0}),x\rangle+\mathbb{E}\left[\langle G(\alpha_{T})g(\alpha_{T}),g(\alpha_{T})\rangle\right]\\
&\qquad-\mathbb{E}\left[\int_{0}^{T}\langle (R+D^{\top}PD)^{-1}(D^{\top}P\sigma-Rr-B^{\top}K-D^{\top}L),D^{\top}P\sigma-Rr-B^{\top}K-D^{\top}L\rangle dt\right]\\
&\qquad+\mathbb{E}\left[\int_{0}^{T}\big(\langle Qq,q\rangle+\langle Rr,r\rangle+\langle P\sigma,\sigma\rangle-2\langle b,K\rangle-2\langle \sigma,L\rangle\big)dt\right]
\end{align*}
and the equality holds when $u(t,i)=v(t,i)$. Now we obtain the optimal control and optimal value.
\end{proof}

\end{document}